\documentclass[12pt]{article}
\usepackage{amsmath,amsxtra,amssymb,color,latexsym,epsfig,amscd,amsthm,subfigure,fancybox,epsfig}
\usepackage[mathscr]{eucal}
\usepackage{graphicx}
\usepackage{cases}
\usepackage{enumitem}

\newtheorem{thm}{Theorem}[section]

\newtheorem{lm}{Lemma}[section]

\theoremstyle{definition}
\newtheorem{defn}[thm]{Definition}

\theoremstyle{remark}

\numberwithin{equation}{section}

\newcommand{\eps}{\varepsilon}

\newcommand{\M}{\mathcal{M}}
\newcommand{\F}{\mathcal{F}}

\newcommand{\E}{\mathbb{E}}

\newcommand{\N}{\mathbb{N}}
\newcommand{\PP}{\mathbb{P}}

\newcommand{\R}{\mathbb{R}}

\newcommand{\Sc}{\mathcal{S}}

\numberwithin{equation}{section}
\newcommand{\1}{\boldsymbol{1}}

\newcommand{\op}{{\cal L}}

\newcommand{\bed}{\begin{displaymath}}
\newcommand{\eed}{\end{displaymath}}
\newcommand{\bea}{\bed\begin{array}{rl}}
\newcommand{\eea}{\end{array}\eed}

\newcommand{\barray}{\begin{array}{ll}}
\newcommand{\earray}{\end{array}}

\def\PM{\text{PM}^\eps}
\def\bar{\overline}
\def\hat{\widehat}
\def\wdt{\widetilde}
\def\a.s{\text{\;a.s.\;}}
\def\bdelta{\boldsymbol{\delta}}
\def\bnu{\boldsymbol{\nu}}

\begin{document}
\title{Long-Run Average Sustainable Harvesting Policies: Near Optimality\thanks{This
research was supported in part by the
Air Force Office of Scientific Research under grant FA9550-15-1-0131.}}

\author{Dang H. Nguyen\thanks{Department of Mathematics, Wayne State University, Detroit, MI
48202,
dangnh.maths@gmail.com.}
\and
 George  Yin\thanks{Department of Mathematics, Wayne State University, Detroit, MI
48202, gyin@math.wayne.edu.}
}
\maketitle

\begin{abstract}
This paper develops near-optimal 
sustainable harvesting strategies for the predator in a predator-prey system. The objective function is of
 long-run average per unit time type.
 To date,
 ecological systems under environmental noise
are usually modeled as
stochastic differential equations driven by a Brownian motion.
Recognizing that the formulation using a
 Brownian motion is
only an idealization, in this paper, it is assumed that
 the environment is subject to
disturbances characterized by a jump process with
rapid jump rates.
Under broad conditions, it is shown that
the systems under consideration can be approximated by a controlled diffusion system.
Based on the limit diffusion system,  control policies of
the original systems are constructed. Such
an approach enables us to develop sustainable harvesting
policies leading to near optimality.
To treat the underlying problems, one of the main difficulties is due to the long-run average objective function.
 This in turn, requires the handling of a number of issues related to ergodicity.  New approaches are developed to obtain the tightness of the underlying processes based on the population dynamic systems.

\bigskip
\noindent {\bf Keywords.} Sustainability; near-optimal strategy;
harvesting policy; long-run-average control;
ergodicity.

\bigskip
\noindent{\bf Subject Classification.}  60H10, 92D25, 93E20.

\end{abstract}

\newpage
\setlength{\baselineskip}{0.22in}
\section{Introduction}\label{sec:int}
This work encompasses the study of controlled predator-prey systems.
The control process is devoted to harvesting activities, which
is one of the central issues in bio-economics.
It has been widely recognized that it may not be a good idea to consider only maximizing short-term benefits focusing purely on harvesting. Although over-harvesting
in a short period may maximize the short-term economic benefits,
 it breaks
the balance between harvesting and its ecological implications.
Thus simple minded policies may lead to detrimental after effect.
As a result, it is crucially important to
pay attention not to render
 exceedingly harmful decision to the environment.
This situation has been observed
in some optimal harvesting models
with finite-time yield or discounted yield; see, e.g., \cite{LA, ALO, LO, SSZ, TY}, among others.

In contrast, ecologists and bio-economists
emphasize the importance of sustainable harvest
in both biological conservation
and long-term economic benefits; see \cite{AS, CC, MM}.
They introduce the concept of {\it maximum sustainable yield}, which is the largest yield (or catch) that can be taken from a species' stock over an infinite horizon.
Their findings
indicate
that
it is more reasonable to maximize the yield in such a way that a species is sustainable and not in danger leading to extinction of the species.
Inspired by the idea of using maximum sustainable yield,
we pay special attentions to sustainability, biodiversity, biological conservation,
and long-term economic benefits,
  and consider long-term horizon optimal strategies in this paper. In lieu of discounted profit, we examine objective functions that are long-run average per unit time type.
As was alluded to, the  papers \cite{LA, ALO, LO, SSZ, TY} concentrated on finite time horizon problems as well as long term objective function under discounting. However, there seems to be not much effort devoted to long-run average criteria
for the harvesting problem to the best of our knowledge.
Discounted objective pays more attention to the current performance, whereas it is certainly necessary to examine the performance when the future is as important.
  This is particularly the case when we take sustainability and long-term economic benefits into consideration.
We consider a long-run-average optimal harvesting problem
for a predator-prey model subject to random perturbations,
in which only the predator takes harvesting action.
This type of optimal harvesting problems
have been studied by some authors; see for example, \cite{LB16, LB14}.
However, harvesting efforts in these
 papers are confined to constant-harvesting strategies only,
which are usually far from optimal
for a larger and more realistic class of harvesting strategies.
In contrast to the discounted criteria, the long-run average criteria are much more difficult to handle.
One of the main difficulties is due to the long-run average cost
 criteria. To treat long-run average objective, one has to handle a number of delicate issues that are related to ergodicity.

To date, ecological systems under environmental noise
are usually modeled by
stochastic differential equations driven by a Brownian motion.
An important aspect of our work is concerned with what if the
noise is not of Brownian motion type. An innovation of the current paper is the use of wideband noise.
It has been widely recognized that
 Brownian motion is
only an idealized formulation or suitable limits of systems in
the real world.
To be more realistic, we would better assume that the environment is subject to
disturbances characterized by a jump process with
rapid jump rates.
This jump process can be modeled by the so-called wideband noise.
Motivated by the approach in \cite{KR},
 we consider a Lotka-Volterra predator-prey model
 with wideband noise and harvesting in this paper. Denote by
 $X^\eps(t)$ and $Y^\eps(t)$
 the sizes of the prey and the predator, respectively.
The system of interest is of the form
\begin{equation}\label{model-2}
\begin{cases}
d X^{\eps}(t)=&X^{\eps}(t)\big[a_1-b_1X^{\eps}(t)-c_1Y^{\eps}(t)\big]dt + \dfrac1\eps X^{\eps}(t)r_1(\xi^\eps(t))dt\\
d Y^{\eps}(t)=&Y^{\eps}(t)\big[a_2-h(Y^{\eps}(t))u(t)-b_2Y^{\eps}(t)+c_2X^{\eps}(t)\big]dt +\dfrac1\eps Y^{\eps}(t)r_2(\xi^\eps(t))dt,
\end{cases}
\end{equation}
where $\eps$ is a small parameter, $\xi(t)$ is an ergodic,
time-homogeneous, Markov-Feller process,
and $\xi^\eps(t)=\xi\left(\frac{t}{\eps^2}\right)$,
$a_i, b_i, c_i, i=1,2$ are positive constants, and
$u(t)$ represent the harvesting effort at time $t$
while $h(\cdot):\R_+\mapsto[0,1]$
indicates the effectiveness of harvesting,
which is assumed to be dependent of the population of the predator.
Thus, the amount of harvested biomass
in a short period of time $\Delta t$
is $Y^{\eps}(t)h(Y^{\eps}(t))u(t)\Delta t$.
Let $\Phi(\cdot):\R_+\mapsto\R_+$
be the revenue function that provides
the economic value as a function
of harvested biomass.
The time-average harvested value
over an interval $[0,T]$
is $\dfrac1T\int_0^T\Phi\Big(h(Y^{\eps}(t)Y^{\eps}(t))u(t)\Big)dt$.
Our goal is to
\begin{equation}\label{opt-har}
\text{ maximize  } \liminf_{T\to\infty}
\dfrac1T\int_0^T\Phi\Big(h(Y^{\eps}(t)Y^{\eps}(t))u(t)\Big)dt \ \hbox{a.s.}
\end{equation}
In our set up,
the harvesting strategy (the control) is only for the predator,  $Y^\eps(\cdot)$, which is also assumed in many papers (e.g., \cite{AHL, BC, MC, XLH}).
The rational is
that the predator has main impacts on the system,
whereas the economic influence of the prey is
not as significant.
In addition,
the prey may be too small or too passive to catch.
Thus we focus on the situation when the control is
in $Y^\eps$ equation only.

Because of the complexity of the model,
developing optimal strategies for the
 controlled system \eqref{model-2} and \eqref{opt-har},
 are usually difficult. Nevertheless,
one may wish to construct policies based on the limit system.
A natural question arises:
Can optimal or near-optimal harvesting strategies
for the diffusion model
be near optimal harvesting strategies
for the wideband-width model
when $\eps$ is sufficiently small?
In a finite horizon,
nearly optimal controls for systems under wideband noise perturbations
were developed in
 the work of Kushner and Ruggaldier \cite{KR}.
As was noted in their paper,
 that the original systems subject to wideband noise perturbations are rather difficult to handle;
there may be additional difficulties if
 the systems are non-Markovian. For infinite horizon problems, it was assumed in \cite{KR}  that the slow and fast components are jointly Markovian.
 By working with the associated probability measures, under suitable conditions, the authors established that there is a limit system being a controlled diffusion process.
Using the optimal or near-optimal controls of the limit systems, one constructs controls for the original systems and show the controls are nearly optimal.
Inspired by their work, we aim to develop near-optimal policies in this paper in an infinite horizon.
We focus on   objective functions  being long-run average per unit time type.
By assuming  the perturbing noise being Markovian,
we develop  near-optimal harvesting strategies (near-optimal controls).
In contrast to optimal controls in a finite horizon, to show that the approximation
works  over an infinite time interval
as in our setting, the ergodicity and the existence of the invariant measure have to be established.
In this paper,
we first show that there exists
an optimal harvesting strategy
for the limit controlled diffusion.
Then, we show that using near-optimal control of the limit diffusion system in the original system leads to  near-optimal controls of
the original system.

We note that in \cite{KR},   nonlinear systems were treated so
a number of assumptions were posed for such wideband noise driven systems in a general setting. In contrast, we have specific systems to work with thus we can no longer posing  general conditions as in the aforementioned paper. Instead, we need to start from scratch.
In fact,
 conditions (C1)-(C4) posed in \cite[Section 7, p. 310]{KR}
include the existence of $\delta$-optimal control, the existence of the associate invariant measure, tightness of the state process, and the  value function under certain admissible class and
the  value function
under stationary admissible relaxed controls being equal.
Because the problems were formulated in a general setting, these conditions are abstract and
are used as
sufficient conditions
to obtain near-optimal controls for wide-band noise
systems. In contrast,
 for the system that we are dealing with,
it is rather difficult to verify these conditions.
Some sufficient conditions
were also proposed in \cite[Conditions (D1)-(D4)]{KR}
by means of a perturbed Lyapunov function method. These conditions
were given to verify conditions (C1)-(C4).
Nevertheless, verifying
conditions (D1)-(D4) in \cite{KR} is still a difficult task for
our model. To begin with, it is difficult to find appropriate Lyapunov functions verifying conditions (D1)-(D4).
To overcome the difficulty,
we propose a new approach rather than finding a function $V$
satisfying the conditions (D1)-(D4) in \cite{KR}.
More precisely,
by analyzing the dynamics of the limit controlled diffusion
when the population of the species is low,
we obtain the tightness of probability measures of the controlled diffusion process.
Then,
using the above as a bridge,
 probabilistic arguments
enable us to prove the tightness
of probability measure of the controlled process perturbed by wideband noise.
Moreover,
we use stochastic analysis to carry out the desired estimates. The analysis itself is new and interest in its own right.
Therefore, the problem arises in control and optimization, but our solution methods are mainly probabilistic.

The rest of the paper is organized as follows.
In Section \ref{sec:form}, we formulate the problem
and identify the limit diffusion system.
The main results are given in Section \ref{sec:main}
while their proofs are provided in Section \ref{sec:pf}.
Section \ref{sec:rem} is devoted to some remarks and possible generalizations.
Finally,
we prove some auxiliary results in an appendix.

\section{Formulation}\label{sec:form}
We work with a complete filtered probability space
$(\Omega,\F,\F_t,\PP)$ satisfying the usual condition.
Denote $\R^2_+=\{(x,y)\in\R^2: x\geq0, y\geq 0\}$
and $\R^{2,\circ}_+=\{(x,y)\in\R^2: x>0, y> 0\}$.
To simplify notations,
we
denote $z=(x,y), \tilde z=(\tilde x,\tilde y),
Z(t)=(X(t),Y(t)), Z^\eps(t)=(X^\eps(t),Y^\eps(t))$.
We assume that harvest efforts
can be represented by a number in a finite interval $\M:=[0,M]$.
Suppose $\xi(t)$
is a pure jump Markov-Feller process
 taking values in a compact metric space $\Sc$.
Suppose its generator is given by
$$Q\phi(w)=q(w)\int_{\Sc}\Lambda(w,d\tilde w)\phi(\tilde w)-q(w)\phi(w)$$
where $q(\cdot)$ is continuous on $\Sc$
and $\Lambda(w, \cdot)$
is a probability measure on $\Sc$ for each $w$.
Suppose that
$\xi(t)$ is uniformly geometric ergodic,
that is
\begin{equation}\label{e1.1a}
\|P(t,w,\cdot)-\bar P(\cdot)\|_{TV}\leq C_0\exp(-\gamma_0 t),\,\text{ for any }\, t\geq 0, w\in\Sc,
\end{equation}
where $\bar P(\cdot)$ is a probability measure in $\Sc$ and $C_0, \gamma_0$ are some positive constants.
Clearly $\bar P(\cdot)$ is an invariant probability measure of $\{\xi(t)\}$.
Let $\chi(w,\cdot)=\int_0^\infty \big[P(t,w,\cdot)-\bar P(\cdot)\big]dt$.
It is well known that
if $\phi(w)$ is a continuous function on $\Sc$ satisfying $\int_{\Sc}\phi(w)\bar P(dw)=0$
then
\begin{equation}\label{e1.1}
\psi(w):=\int_{\Sc}\chi (w,d\tilde w)\phi(\tilde w)\,\text{
satisfying }\,Q\psi(w)=-\phi(w).
\end{equation}
Note that $\psi(\cdot)$ is well defined thanks to
the exponential decay in \eqref{e1.1a}.
Suppose that
\begin{equation}\label{e1.2}
r_i(\cdot)\,\text{ is bounded in }\, \Sc,\text{ and }\, \int_{\Sc}r_i(w)\bar P(dw)=0, \ i=1,2.
\end{equation}
Let
$A=(a_{ij})_{2\times2}$ with
$$a_{ij}=\int_{\Sc}\int_{\Sc}\chi(w,d\tilde w)\bar P(dw)\Big[r_i(w)r_j(\tilde w)+r_j(w)r_i(\tilde w)\Big].$$
We suppose that $A$ is positive definite with square root $(\sigma_{ij})_{2\times2}.$
Consider the diffusion
\begin{equation}\label{model-3}
\begin{cases}
d X(t)=X(t)\big[\bar a_1-b_1X(t)-c_1Y(t)\big]dt + X(t)(\sigma_{11}dW_1(t)+\sigma_{12}dW_2(t))\\
d Y(t)=Y(t)\big[\bar a_2-h(Y(t))u(t)-b_2Y(t)+c_2X(t)\big]dt + Y(t)(\sigma_{12}dW_1(t)+\sigma_{22}dW_2(t)),
\end{cases}
\end{equation}
	where
$\bar a_1=a_1+\dfrac{a_{11}}2=a_1+\dfrac{\sigma_{11}^2+\sigma_{12}^2}2$,
$\bar a_2=a_2+\dfrac{a_{22}}2=a_2+\dfrac{\sigma_{22}^2+\sigma_{12}^2}2$,
$W_1, W_2$ are two independent Brownian motions.

We suppose that the function $\Phi(\cdot):\R_+\mapsto\R_+$ represents the yield that  is
 Lipschitz in its argument satisfying  $\Phi(0)=0$.
That is, the yield is zero if
we harvest nothing.
If we want to maximize the average amount of the species harvested,
then $\Phi(y)=y$.
If we want to maximize the average money earned,
$\Phi(y)$ should have a ``saturated'' form, such as
$\Phi(y)=\dfrac{y}{c+y}$.
We assume the effectiveness $h(\cdot):\R_+\mapsto[0,1]$
is an increasing function and $h(0)=0$.
This
stems from
that the effectiveness
increases as the density of the species increases.

Let $\PM$ be the class of functions
$v:\R^2_+\times\Sc\mapsto\M$ such that
under the feedback control $u(t)=v(Z^\eps(t))$
there exists  a solution process to \eqref{model-2},
which is a Markov-Feller process.
For $v\in\PM$, define
$$J^{\eps}
(v):=\liminf_{T\to\infty}\dfrac1T\int_0^T\Phi
\Big(h(Y^{\eps}(t)Y^{\eps}(t))v(Y^\eps(t))\Big)dt \ \hbox{ a.s.}$$
For the wideband noise system,
it is difficult to find an optimal control, that is, a control $v^*\in\PM$ satisfying
$${\mathfrak J}^\eps
=\sup_{v\in\PM}\{J^{\eps}(v)
\}.$$
Thus, our goal is to find a near-optimal control $v\in\PM$
using the limit diffusion system.
To do that, we broaden the class of controls
by use of the ``relaxed controls".

We present here some concepts and notation introduced in \cite{KR}.
Let $M(\infty)$
denote the family of measures $\{m(\cdot)\}$ on the Borel subsets of
$[0,\infty)\times U$
satisfying $m([0,t]\times U)=t$
for all $t\geq0$.
By
the weak convergence $m_n(\cdot)\rightarrow m(\cdot)$
in $M(\infty)$
we mean $\lim_{n\to\infty}\int f(s,\alpha)m_n(ds\times d\alpha)
=\int f(s,\alpha)m(ds\times d\alpha)
$
for any continuous function $f(\cdot):[0,\infty)\times U\mapsto\R$ with compact support.

A random measure $m(\cdot)$ with values in $M(\infty)$ is said to be an admissible relaxed control for \eqref{model-2}
if $\int_U\int_0^tf(s,\alpha)m(ds\times d\alpha)$
is progressively measurable with respect to $\F^\eps_t:=\F_{\frac{t}\eps}$ for each bounded continuous function $f(\cdot)$.
With a relaxed control $m(\cdot)$,
let $\bar m_t=\lim_{s\to t} \dfrac1{s-t}\int_\M\int_t^sm(ds\times d\alpha)$,
the model \eqref{model-2} becomes
\begin{equation}\label{model-4}
\begin{cases}
d X^{\eps}(t)=&X^{\eps}(t)\big[a_1-b_1X^{\eps}(t)-c_1Y^{\eps}(t)\big]dt + \dfrac1\eps X^{\eps}(t)r_1(\xi^\eps(t))dt\\
d Y^{\eps}(t)=&Y^{\eps}(t)\big[a_2-h(Y^{\eps}(t))\bar m_t-b_2Y^{\eps}(t)+c_2X^{\eps}(t)\big]dt + \dfrac1\eps Y^{\eps}(t)r_2(\xi^\eps(t))dt
\end{cases}
\end{equation}
Let $\mathcal P(\M)$ be the space of invariant probability measures
with Prohorov's topology.
A relaxed control is said to be Markov
if there exists a measurable function $v:\R^2_+\mapsto\mathcal P(\M)$
such that $m_t=v(X^\eps(t)), t\geq0.$
For $z\in\R^2_+$, $w\in\Sc$ and $u\in\M$, define
$$
F(z,w)=\Big(xr_1(w), yr_2(w)\Big)^\top
$$
and
$$
G(z, u)=
\Big(x[a_1-b_1x-c_1y], y[a_2-h(y)u-b_2y+c_2x]\Big)^\top
.$$
By an ergodicity argument
(see, for example,
\cite{DN,DNY,HNY,RR}),
it can be shown that
if $-a_2+c_2\dfrac{a_1}{b_1}<0$ then
for any admissible control $u(t)$,
$Y^\eps(t)$ tends to $0$ with probability 1,
which implies
$$\lim_{T\to\infty}\dfrac1T\int_0^T\Phi\Big(h(Y^{\eps}(t))Y^{\eps}(t)u(t)\Big)dt=0 \text{ a.s.}$$
Thus, to avoid the trivial limit,
we assume throughout this paper that
\begin{equation}\label{positive}
-a_2+c_2\dfrac{a_1}{b_1}>0.
\end{equation}

Define the operator
$$\op^\eps_u \phi(z, w)
=\dfrac1{\eps^2}Q\phi(z,w)+\dfrac1\eps \dfrac{\partial \phi(z,w)}{\partial z}F(z,w)+\dfrac{\phi(z,w)}{\partial z}G(z, u),
$$
where
$\phi:\R^2_+\times\Sc\mapsto \R$ is continuous and
have continuous derivative with respect to the first variable,
$\dfrac{\partial \phi(z,w)}{\partial z}$.
Denote by $\PP_{z,w}$ and $\E_{z,w}$
the probability measure and the corresponding expectation
of the process $(Z^\eps(\cdot),\xi^\eps(\cdot))$
with initial condition $(z,w)$.
Note that $\PP_{z,w}$ and $\E_{z,w}$
depends implicitly on the control $m(t)$.
For any bounded stopping times $\tau_1\leq\tau_2$,
we have
$$\E_{z,w}\phi(Z^\eps(\tau_2),\xi^\eps(\tau_2))
=\E_{z,w}\phi(Z^\eps(\tau_1),\xi^\eps(\tau_1))
+\E_{z,w}\int_{\tau_1}^{\tau_2}\op^\eps_{\bar m_s}\phi(Z^\eps(s),\xi^\eps(s))ds
$$
given that the expectations involved exist.

A random measure $m(\cdot)$ with values in $M(\infty)$ is said to be an admissible relaxed control for \eqref{model-3}
if $\int_U\int_0^tf(s,\alpha)m(ds\times d\alpha)$
is independent of $\{W_i(t+s)-W_i(t), s>0, i=1,2\}$ for each bounded continuous function $f(\cdot)$.
Under a relaxed control $m(\cdot)$, the controlled diffusion \eqref{model-3} becomes
\begin{equation}\label{model-5}
\begin{cases}
d X(t)=X(t)\big[\bar a_1-b_1X(t)-c_1Y(t)\big]dt + X(t)(\sigma_{11}dW_1(t)+\sigma_{12}dW_2(t))\\
d Y(t)=Y(t)\big[\bar a_2-h(Y(t))\bar m_t-b_2Y(t)+c_2X(t)\big]dt + Y(t)(\sigma_{12}dW_1(t)+\sigma_{22}dW_2(t)).
\end{cases}
\end{equation}
The generator for the controlled diffusion process \eqref{model-5} is
$$
\begin{aligned}
\op_u\phi(z)=&\dfrac{\partial\phi(z)}{\partial x}x[\bar a_1-b_1x-c_1y]+\dfrac{\partial\phi(z)}{\partial x} y[\bar a_2-h(y)u-b_2y+c_2x]\\
&+\dfrac12\left(a_{11}\dfrac{\partial^2\phi(z)}{\partial x^2}x^2+2a_{12}\dfrac{\partial^2\phi(z)}{\partial x\partial y}xy+a_{22}\dfrac{\partial^2\phi(z)}{\partial y^2}y^2\right).
\end{aligned}
$$

\begin{defn}{\rm
A relaxed control $m(\cdot)$ for \eqref{model-5} is said to be Markov
if there exists a measurable function $v:\R^2_+\mapsto\mathcal P(\M)$
such that $m_t=v(Z(t)), t\geq0.$
A Markov control $v$ is a relaxed control satisfying that
$v(z)$ is a Dirac measure on $\M$ for each $z\in\R^2_+$.
Denote the set of Markov controls
and relaxed Markov controls by $\Pi_{M}$ an $\Pi_{RM}$, respectively.
With a relaxed Markov control, $Z(t)$
is a Markov process that has the strong Feller property in $\R^{2,\circ}_+$;
see \cite[Theorem 2.2.12]{ABG}.
Since the diffusion is nondegenerate in $\R^{2,\circ}_+$,
if the process $Z(t)$ has an invariant probability measure in $\R^{2,\circ}_+$,
the invariant measure is unique, denoted by $\eta_v$.
In this case, the control $v$ is said to be stable.}\end{defn}

\section{Main Results}\label{sec:main}
First, we need the existence and uniqueness of positive solutions
to \eqref{model-5} for any admissible relaxed control.

\begin{lm}\label{lm3.1}
If $m(\cdot)$ is an admissible relaxed control for \eqref{model-3} $($or \eqref{model-5}$)$,
then there exists a unique nonanticipative solution to
\eqref{model-5} with initial value $z=(x,y)\in\R^2_+$ satisfying
\begin{enumerate}
\item $\PP_{z}\{X(t)>0, \ t\geq 0\}=1$
$($resp. $\PP_{z}\{X(t)>0, \ t\geq 0\}=1)$
if $x>0$ $($resp. $y>0)$,
and
$\PP_{z}\{X(t)=0\, t\geq 0\}=1$
$($resp. $\PP_{z}\{Y(t)=0\, t\geq 0\}=1)$
if $x=0$ $($resp. $y=0)$.
\item
$$
\E_{z} \sup_{t\leq T}(|Z(t)|^2)\leq K(1+|z|^2)
$$
where $K$ depends only on $T$.
\end{enumerate}
\end{lm}

\begin{proof}
This lemma can be proved by arguments in \cite[Theorem 1]{KR}
or \cite[Theorem 2.2.2]{ABG}.
Note that the coefficients in \eqref{model-5} do not satisfy the
 linear growth condition.
However, using a truncation argument
and a Khaminskii-type method in \cite{LM},
we can easily prove the existence of a unique solution to \eqref{model-5}
satisfying claim 1.
Moreover, we can estimate
$$
\begin{aligned}
d[c_2X(t)+c_1Y(t)]=&c_2X(t)\big[\bar a_1-b_1X(t)\big]dt +c_1Y(t)\big[\bar a_2-h(Y(t))\bar m_t-b_2Y(t)\big]dt \\
&+c_2X(t)(\sigma_{11}dW_1(t)+\sigma_{12}dW_2(t))
 + c_1Y(t)(\sigma_{12}dW_1(t)+\sigma_{22}dW_2(t))\\
\leq& [c_2\bar a_1X(t)+c_1\bar a_2X(t)]dt
\\&+c_2X(t)(\sigma_{11}dW_1(t)+\sigma_{12}dW_2(t))+ c_1Y(t)(\sigma_{12}dW_1(t)+\sigma_{22}dW_2(t)).
\end{aligned}
$$
In this estimate,
the right-hand side is linear in $X(t)$ and $Y(t)$.
Using standard arguments, (e.g., \cite[Theorem 3.5]{RK} or \cite[Proposition 3.5]{ZY}),
we can obtain the moment estimate, the second claim of this lemma.
\end{proof}

With this lemma, in each finite interval,
we can approximate $Z^\eps(t)$
by $Z(t)$, which is proved in \cite[Theorem 5]{KR}.

\begin{lm}\label{lm2.2}
For any compact set $\mathcal K\in\R^2_+$,
$\{(Z^\eps(\cdot),m(\cdot)), t\geq 0\}$ with $Z^\eps(0)\in\mathcal K$
is tight in $D[0,\infty)\times M(\infty)$.
If $(Z^{\eps_k}(\cdot),m^{(k)}(\cdot))$
converges weakly to $(\hat Z(\cdot),\hat m(\cdot))$ as $k\to\infty$
with $\eps_k\to0$ as $k\to\infty$,
then there exists independent Brownian motions $W_1(t)$ and $W_2(t)$
such that $\hat m(\cdot)$ is progressively measurable with respect to the filtration generated by $W_1(t), W_2(t)$,
and
$\hat Z$ satisfying \eqref{model-5}
with $(Z(\cdot),m(\cdot))$ replaced by $(\hat Z(\cdot),\hat m(\cdot))$.
\end{lm}

We need the following lemma, whose proof is postponed to
the appendix.

\begin{lm}\label{lm2.3}
The following claims hold.
\begin{itemize}
\item For any admissible relaxed control $m(\cdot)$,
we have that
\begin{equation}\label{e1-lm2.3}
\limsup_{t\to\infty} \dfrac1T\int_0^T \Phi\Big(Y(t)h(Y(t))\bar m_t\Big)dt\leq \wdt C \ \hbox{ a.s. }
\end{equation}
for some constant $\wdt C$.

\item Every relaxed Markov control is stable and
there exists $\hat C>0$ such that
\begin{equation}\label{e2-lm2.3}
\int_{\R^{2,\circ}\times U} [1+\Phi(yh(y)u)]^2\pi_v(dz\times du)\leq \hat C
\end{equation}
for any relaxed Markov control $v$,
where
$\pi$ is a measure in $\R^{2,\circ}_+\times\M$
defined by
$$\pi_v(dz\times du)=[v(z)(du)]\times \eta_v(du).$$
\item The family $\{\eta_\nu:\,\nu\in\Pi_{RM}\}$
is tight in $\R^{2,\circ}_+$.
$[$Recall that $\eta_\nu$ is the invariant measure.$]$
\end{itemize}
\end{lm}

With this lemma, letting
$$\rho^*=\sup_{v\in\Pi_{RM}}\left\{\int_{\R^{2,\circ}\times U} \Phi(yh(y)u)\pi_v(dz\times du)\right\},$$
we have the following result from \cite[Theorems 3.7.11 and 3.7.14]{ABG}.

\begin{thm}\label{thm2.1}
The  Hamilton-Jacobi-Bellman $($HJB$)$ equation
$$\max_{u\in U}\Big[\op_u V(x)+c(x,u)\Big]=\rho$$
admits a solution $V^*\in C^2(\R^{2,\circ}_+)$ satisfying $V^*(0)=0$
and $\rho=\rho^*$.
A relaxed Markov control is optimal if and only if
it satisfies
$$
\begin{aligned}&\!\!\!
\dfrac{\partial V^*}{\partial y}\Big[y(-a_2-h(y)\bar v(z) -b_2y+c_2x]+\Phi(yh(y)v(z))\Big]\\
&\qquad =\max_{u\in U} \dfrac{\partial V^*}{\partial y}\Big[y(-a_2-h(y)u-b_2y+c_2x]+\Phi(yh(y)u)\Big],
\end{aligned}
$$
where $\bar v(z)=\int_\M u[v(z)(du)].$
\end{thm}

The existence of an optimal Markov control can be
derived from a well-known selection theorem; see e.g., \cite[pp. 199-200]{FR75}.
Let $v^*$ be an optimal Markov control.
There exists a sequence of $v_n:\R^2_+\mapsto U$
such that $v_n(z)$ is locally Lipschitz in $z$
and $\lim_{n\to\infty} v_n=v$ almost everywhere in $\R^{2,\circ}_+$.
Since every Markov control is stable,
and the family $\{\nu_v, v\in \Pi_{RM}\}$
is tight on $\R^{2,\circ}_+$,
we have from \cite[Lemma 3.2.6]{ABG}
that
\begin{equation}\label{e-delta.opt}
\lim_{n\to\infty} \rho_{\nu_n}=\rho_{v^*}=\rho^*.
\end{equation}
This indicates that
we can always find a $\delta$-optimal Markov control
that is locally Lipschitz.
We state here the main result of this paper.

\begin{thm}\label{thm2.2}
For any $\delta>0$, there exists a locally Lipschitz Markov control $u^\delta$
such that
$$J^\eps(u^\delta):=\liminf_{T\to\infty}\dfrac1T\int_0^T\Phi\Big(h(Y^{\eps}(t)Y^{\eps}(t))u^\delta(t)\Big)dt\leq \rho^*+\delta$$
and that for sufficiently small $\eps>0$,
we have
$$J^\eps(u^\delta)\geq \mathfrak{J^\eps}-3\delta.$$
\end{thm}

The result above is known as chattering-type theorem. It connects relaxed controls and that of ordinary controls, and indicates that for any relaxed control, we can find a locally Lipschitz control to approximate the relaxed control. This is important because even though relaxed controls facilitate the establishment of the desired asymptotic results. Such control sets are much larger than the usual ordinary controls and cannot be used in the real applications. Thus viable approximation will be much appreciated.
In view of \cite[Theorem 8]{KR},
 we proceed to verify the following conditions to prove the desired result.

\begin{enumerate}[label=(\rm C{\arabic*})]
    \item \label{C1}
    There is an $\eps_0>0$ such that $\{Z^\eps(u,t), u\in {PM}^\eps, 0\leq t<\infty, \eps\leq \eps_0\}$ is $\PP_{z,w}$-tight in $\R^{2,\circ}_+$
    for each $(z,w)\in\R^{2,\circ}_+\times\Sc$.
\item \label{C2}
There is a $\delta$-optimal Markov control $u(z)$
that is locally Lipschitz in $z$ for any $\delta>0$.
\end{enumerate}

Condition \ref{C2} has been verified in our manuscript;  see \eqref{e-delta.opt}.
Since the dynamics of $Z^\eps(t)$ is
dominated
by negative quadratic terms
when $Z^\eps(t)$ is large,
it is easy to prove the tightness of  $\{Z^\eps(u,t), u\in {PM}^\eps, 0\leq t
<\infty, \eps\leq \eps_0\}$ in $\R^{2}_+$.
However, we need the tightness in $\R^{2,\circ}_+$ to achieve the near optimality.
To do that we need to analyze the behavior of $Z^\eps(u,t)$
near the boundary.
Inspired by \cite{BL},
we utilize the ergodicity of the system on the boundary and
a property of the Laplace transform
 to construct a function
$V^\eps(z,w)$
satisfying the inf-compact condition in $\R^{2,\circ}_+$, i.e.,
$$\lim_{R\to\infty}\inf\left\{V^\eps(z,w): z+\frac1x+\frac1y>R\right\}=\infty$$
and that
$$\E_{z,w} V^\eps(Z^\eps(t),\xi^\eps(t))\leq C(1+V(z,w))$$
for any control $u\in {PM}^\eps$ and $t\geq0$.
Clearly, \ref{C1} is proved if such a function is constructed.
In contrast to
the technique used in \cite{BL},
which is applied to a process in a compact space,
the verification in our case is more difficult
because the space $\R^2_+$ is not compact
and we have to treat a family of singularly perturbed processes
rather than a single process.

\section{Proofs of Results}\label{sec:pf}
First, when  $p_0$, $p_1$, $p_2>0$
are sufficiently small, we have
\begin{equation}\label{e3.1}
2p_0+p_1b_1+p_2c_2<b_1,\,\text{ and }\, 2p_0+ p_1c_1+p_2b_2<c_1.
\end{equation}
We can also choose $p_1$ and $p_2$ such that
\begin{equation}\label{e3.2}
p_1a_1-p_2a_2<0.
\end{equation}
By \eqref{positive} and \eqref{e3.2},
we have
\begin{equation}\label{e-lambda}\lambda=\dfrac{1}{11}\min\left\{p_1a_1-p_2a_2, p_2\left(-a_2+\dfrac{a_1c_2}{b_1}\right)\right\}>0
\end{equation}

In view of \eqref{e1.1} and \eqref{e1.2}, there exist
bounded functions
 $r_3(w)$
and $r_4(w)$
such that
$Qr_3(w)=r_1(w)$
and $Qr_4(w)=r_2(w)$.
Let $V(x,y)=\dfrac{1+c_2x+c_1y}{x^{p_1}y^{p_2}}.$
Define
$$
V_1(z,w):=xr_3(w)\dfrac{\partial V(z)}{\partial x}+yr_4(w)\dfrac{\partial V(z)}{\partial y}.
$$
We have
\begin{equation}\label{e3.8}
QV_1(z,w)=-xr_1(w)\dfrac{\partial V(z)}{\partial x}-yr_2(w)\dfrac{\partial V(z)}{\partial y}=-\dfrac{\partial V(z)}{\partial z}\cdot F(z,w).
\end{equation}
By direct calculation and the boundedness of $r_i(w)$,  for $i=3,4$, there is a $K_2>0$ such that
\begin{equation}\label{e3.3}
\left|V_1(z,w)\right|\leq K_2V(z),\, (z,w)\in\R^{2,\circ}_+\times\Sc,
\end{equation}
\begin{equation}\label{e3.4}\left|\dfrac{\partial V_1(z,w)}{\partial z}\cdot F(z,w)\right|\leq K_2V(z),\, (z,w)\in\R^{2,\circ}_+\times\Sc,
\end{equation}
and
\begin{equation}\label{e3.5}\left|\dfrac{\partial V_1(z,w)}{\partial z}\cdot G(z,u)\right|\leq K_2(1+|z|)V(z),\, (z,w)\in\R^{2,\circ}_+\times\Sc, u\in\M.
\end{equation}

In view of \eqref{e3.1}, there exists an $H>0$ such that
$$
\begin{aligned}
\inf_{z\in\R^2_+,|z|>H,u\in\M}\bigg\{p_1&\big|a_1-b_1x-c_1y\big|+p_2\big|-a_2-h(y)u-b_2y+c_2x\big|\\
&+3+K_2+p_0(1+|z|)+\dfrac{c_2x(a_1-b_1x)+c_1y(-a_2-h(y)u-b_2y)}{1+c_2x+c_1y}\bigg\}<0.
\end{aligned}
$$
Let
$$
\begin{aligned}
H_1:=\sup_{z\in\R^2_+,|z|\leq H,u\in\M}\bigg\{&p_1\big|a_1-b_1x-c_1y\big|+p_2\big|-a_2-h(y)u-b_2y+c_2x\big|+3+K_2\\
&+p_0(1+|z|)+\dfrac{c_2(a_1-b_1x)
+c_1(-a_2-h(y)u-b_2y)}{1+c_2x+c_1y}\bigg\}<\infty.
\end{aligned}
$$
By the definitions of $H$ and $H_1$, we have
\begin{equation}\label{e3.9}
\begin{aligned}
\dfrac{\partial V(z)}{\partial z}\cdot G(z, u)
=&V(x,y)\bigg[-p_1\big(a_1-b_1x-c_1y\big)-p_2\big(-a_2-h(y)u-b_2y+c_2x\big)\\
&\qquad\qquad+\dfrac{c_2(a_1-b_1x)+c_1(-a_2-h(y)u-b_2y)}{1+c_2x+c_1y}\bigg]\\
\leq&\big(H_1\1_{\{z<H\}}-3-K_2-p_0(1+|z|)\big) V(z).
\end{aligned}
\end{equation}

Let $V^\eps(z,w)=V(z)+\eps V_1(z,w),$
we have from \eqref{e3.3} that
\begin{equation}\label{e3.6}
(1-\eps K_2)V(z)\leq V^\eps(z,w)\leq (1+\eps K_2)V(z),\,\,\, z\in\R^{2,\circ}_+, s\in\Sc.
\end{equation}
If $\eps>0$ is sufficiently small such that
\begin{equation}\label{e3.10}
\eps K_2\leq p_0;\,\, (H_1+3)\eps K_2<1;\,\,
\end{equation}
using \eqref{e3.4}, \eqref{e3.5}, \eqref{e3.8}, and \eqref{e3.9},
we can estimate
\begin{equation}\label{e3.7}
\begin{aligned}
\op^\eps_u V^\eps(z,w)=&\dfrac{\partial V(z)}{\partial z}\left[\dfrac1\eps F(z,w)+G(z, u)\right]\\
&+\eps\dfrac{\partial V_1(z, w)}{\partial z}\left[\dfrac1\eps F(z,w)+G(z, m)\right]+\dfrac1{\eps}Q V_1(z,w)
\\
\leq& \big(H_1\1_{\{z<H\}}-3-K_2-p_0(1+|z|)\big) V(z)+K_2V(z)+\eps K_2(1+|z|)V(z)\\
\leq& \big((H_1+1)\1_{\{z<H\}}-2)V(z)\\
\leq&\big((H_1+2)\1_{\{z<H\}}-1)V^\eps(z,w),
\end{aligned}
\end{equation}
where the last two lines follow from \eqref{e3.6} and \eqref{e3.10}.
By virtue of \eqref{e3.7},
standard arguments show that
\begin{equation}\label{e3.11}
\E_{z,w} V^\eps(Z(t))\leq e^{(H_1+2)t}V^\eps(z),\, t\geq0, z\in\R^{2,\circ}_+, w\in\Sc
\end{equation}
Let $\tau^\eps=\inf\{s\geq 0: Z^\eps(s)\leq H\}$. Since $\op^\eps_u V^\eps(z,w)\leq -V^\eps(z,w)$ if $z\geq H$,
we have that
\begin{equation}\label{e3.12}
\begin{aligned}
\E_{z,w} e^{t\wedge\tau^\eps}V^\eps\big(Z^\eps(t\wedge\tau^\eps),\xi^\eps(t\wedge\tau^\eps)\big)
&=V^\eps(z)+\E_{z,w}\int_0^{t\wedge\tau^\eps}e^{s}\Big[V^\eps(s)+\op^\eps_{\bar m_s}V^\eps(Z^\eps(s),\xi^\eps(s))\Big]ds\\
&\leq V^\eps(z),\,\text{ for }\, t\geq0, z\in\R^{2,\circ}_+, w\in\Sc.
\end{aligned}
\end{equation}

\begin{lm}\label{lm3.2}
There exist $L>0$ and $\eps_1>0$
such that for all $\eps<\eps_1$,
\begin{equation}\label{e3.18}
\E_{z,w} \dfrac{1}{V_1^\eps(Z^\eps(t),\xi^\eps(t))}\leq Le^{(H_1+2)t}\dfrac{1+|z|^2}{V_1^\eps(z,w)},\,\text{ for }\, (z,w)\in\R^{2,\circ}_+\times\Sc, \ t\geq0.
\end{equation}
\end{lm}

\begin{proof}
Let $\tilde V(z)=(1+c_2x+c_1y)x^{p_1}y^{p_2}$.
Construct a perturbed Lyapunov function
$$\tilde V^\eps(z,w)=\tilde V(z)+\eps\left(xr_3(w)\dfrac{\partial\tilde V(z)}{\partial x}+yr_4(w)\dfrac{\partial\tilde V(z)}{\partial y}\right)
$$
 Similar to estimates in \eqref{e3.7},
we can find $K_3>0$ such that
\begin{equation}\label{e3.13}
(1-\eps K_3)\tilde V(z)\leq\tilde V^\eps(z,w)\leq (1+\eps K_3)\tilde V(z)
\end{equation}
and
\begin{equation}\label{e3.14}
\E_{z,w} \tilde V^\eps(Z^\eps(t), \xi^\eps(t))\leq e^{(H_1+2)t}\tilde V^\eps(z,w)
\end{equation}
when $\eps$ is sufficiently small.
On the other hand, for any $(z,w)\in\R^{2,\circ}_+\times\Sc$, we have
\begin{equation}\label{e3.15}
\dfrac1{V(z)}\leq \tilde V(z)\leq  (1+c_2x+c_1y)^2 \dfrac1{V(z)}.
\end{equation}
which combined with \eqref{e3.6} and \eqref{e3.13} implies that
\begin{equation}\label{e3.16}
\dfrac{1}{V_1^\eps(z,w)}\leq \dfrac{1}{(1-\eps K_2)V(z)}\leq  \dfrac{1}{(1-\eps K_2)(1-\eps K_3)}\tilde V^\eps(z,w)
\end{equation}
and
\begin{equation}\label{e3.17}
\tilde V^\eps(z,w)
\leq  (1+\eps K_3)\tilde V(z)\leq (1+\eps K_3) \dfrac{(1+c_2x+c_1y)^2}{V(z)}\leq (1+\eps K_2)^2 \dfrac{(1+c_2x+c_1y)^2}{V^\eps(z,w)}
.\end{equation}
Applying \eqref{e3.16} and \eqref{e3.17} to \eqref{e3.14},
 we can
easily obtain \eqref{e3.18}
for suitable $L>0$
when $\eps$ is sufficiently small.
\end{proof}

\begin{lm}\label{lm3.3}
There are $\hat K>0$ and $\eps_2>0$ such that
for any  $\eps<\eps_2$, and any admissible control
$m(\cdot)$ for \eqref{model-2}, we have
$$\E_{z,w}\int_0^t |Z^\eps(s)|^2ds\leq \hat K(1+|z|+t),$$
and
$$\E_{z}\int_0^t |Z(s)|^2ds \leq \hat K(1+|z|+t).$$
\end{lm}

\begin{proof}
Let $V_2(z)=1+c_2x+c_1y$
and
$$
V_3(z,w):=xr_3(w)\dfrac{\partial V_2(z)}{\partial x}+yr_4(w)\dfrac{\partial V_2(z)}{\partial y}.
$$
We can find a $K_4>0$ satisfying
\begin{equation}\label{e5-lm3.3}
\left|V_3(z,w)\right|\leq K_4V_2(z),\, (z,w)\in\R^{2,\circ}_+\times\Sc.
\end{equation}
and
\begin{equation}\label{e6-lm3.3}
\left|\dfrac{\partial V_3(z,w)}{\partial z}\cdot F(z,w)\right|\leq K_4V_2(z),\, (z,w)\in\R^{2,\circ}_+\times\Sc.
\end{equation}
\begin{equation}\label{e7-lm3.3}
\left|\dfrac{\partial V_3(z,w)}{\partial z}\cdot G(z,u)\right|\leq K_4(1+|z|)V_2(z),\, (z,w)\in\R^{2,\circ}_+\times\Sc, u\in\M.
\end{equation}
We have
$$\dfrac{\partial V_2(z)}{\partial z}\cdot F(z,u)=
c_2x\big[a_1-b_1x\big] +c_1y\big[a_2-h(y)u-b_2y\big].$$
Let $\beta\in\big(0, (c_2b_1)\wedge(c_1b_2)\big)$.
Clearly, we can choose a $K_5>0$ such that
\begin{equation}\label{e4-lm3.3}
\dfrac{\partial V_2(z)}{\partial z}\cdot F(z,u)\leq K_5-V_2(z)-\beta(x^2+y^2)\,\forall (x,y)\in\R^2_+, u\in[0,M].
\end{equation}
Let
$$V_2^\eps(z,w)=V_2(z)+\eps V_3(z,w)$$
Similar to \eqref{e3.7},
 from \eqref{e5-lm3.3}, \eqref{e6-lm3.3}, and \eqref{e7-lm3.3}, we have
$$\op^\eps_u V^\eps_3(z,w)\leq 2K_5-\dfrac{\beta
|z|^2}2$$
for sufficiently small $\eps$.
As a result,
\begin{equation}\label{e3-lm3.3}
\begin{aligned}
\E_{z,w} V^\eps_3(Z^\eps(t),\xi^\eps(t))=&V^\eps_3(z,w)+\E_{z,w}\int_0^t\op^\eps_{\bar m_t} V^\eps_3(Z^\eps(ds),\xi^\eps(s))ds\\
\leq&V^\eps_3(z,w)+2K_5t-\dfrac\beta2\int_0^t\E_{z,w} |Z^\eps(t)|^2,
\end{aligned}
\end{equation}
which leads to
$$
\dfrac\beta2\int_0^t\E_{z,w} |Z^\eps(t)|^2
\leq V^\eps_3(z,w)+2K_5t
$$
The first claim of the lemma follows directly from the above estimate.
The second claim can be derived by
applying It\^o's formula for $V_2(z)$ to \eqref{model-5}
and then proceeding like \eqref{e3-lm3.3}.
\end{proof}

\begin{lm}\label{lm3.4}
There is a $\tilde K>0$ such that
$$\bigg|\E_{z,w}\big[\ln V(Z^\eps(T))\big]-\ln V(z)-
\E_{z,w}\int_0^T\op_{\bar m_t}\ln V(Z^\eps(t), \xi^\eps(t))dt\bigg|\leq\tilde K(1+T)\eps.$$
for any admissible relaxed control $m(\cdot)$.
\end{lm}

\begin{proof}
Let $$g_1(z,w)=\int_\Sc\chi(w, d\tilde w)\dfrac{\partial(\ln V(z))}{\partial z}\cdot F(z,\tilde w),$$
and
$$g_2(z,w)=\int_\Sc\chi(w, d\tilde w)\left[\dfrac{\partial g_1(z,w)}{\partial z}F(x, \tilde w)+\dfrac{\partial(\ln V(z))}{\partial z}\cdot G(x, u)-\op_u \ln V(z)\right].$$
Note that
$g_2$ does not depend on $u$ since there is no $u$ dependence in
$$\dfrac{\partial(\ln V(z))}{\partial z}\cdot G(x, u)-\op_u \ln V(z)=\dfrac12\dfrac{a_{11}c_2^2x^2+a_{22}c_1^2y^2+2a_{12}c_1c_2xy}{(1+c_2x+c_1y)^2}-\dfrac{c_2xa_{11}+c_1y a_{22}}{1+c_2x+c_1y}.
$$
Moreover, direct calculations show that
$\dfrac{\partial(\ln V(z))}{\partial z}\cdot F(z,w)
$
and
$\dfrac{g_1(z,w)}{\partial z}\cdot F(x,  y)$ are bounded along with $\dfrac{\partial(\ln V(z))}{\partial z}\cdot G(x, u)-\op_u \ln V(z)$.
Consequently,
$g_i(z,w), i=1,2$ are also bounded in $\R^{2,\circ}_+\times\Sc$.
As a result,
we have from \cite[Formula (4.21)]{BP} that
$$
\Big|\op^\eps_u[\ln V(z)+\eps g_1(z,w)+\eps^2g_2(z,w)]
-\op_u \ln V(z)\Big|\leq K_6\eps\,\text{ for all }\, (z,w)\in\R^{2,\circ}_+\times\Sc
$$
for some constant $K_6>0$ independent of $m$.
Combining this and the equality
$$
\begin{aligned}
\E_{z,w}&\big[\ln V(Z^\eps(T))+\eps g_1(Z^\eps(T),\xi^\eps(T))+\eps^2g_2(Z^\eps(T),\xi^\eps(T))\big]\\
=&\ln V(z)+\eps g_1(z,w)+\eps^2g_2(z,w)\\
&+\E_{z,w}\int_0^T \op^\eps_{\bar m_t}[\ln V(Z^\eps(t))+\eps g_1(Z^\eps(t),\xi^\eps(t))+\eps^2g_2(Z^\eps(t),\xi^\eps(t))]dt,
\end{aligned}
$$
we obtain
$$
\begin{aligned}
\bigg|&\E_{z,w}\big[\ln V(Z^\eps(T))+\eps g_1(Z^\eps(T),\xi^\eps(T))+\eps^2g_2(Z^\eps(T),\xi^\eps(t))\big]\\
&-\ln V(z)-\eps g_1(z,w)-\eps^2g_2(z,w)-\E_{z,w}\int_0^T\op_{\bar m_t}[\ln V(Z^\eps(t), \xi^\eps(t))]dt\bigg|\leq K_6T\eps.
\end{aligned}
$$
By the boundedness of $g_i(z,w), i=1,2$, we deduce that
$$\bigg|\E_{z,w}\big[\ln V(Z^\eps(T))\big]-\ln V(z)-\E_{z,w}\int_0^T\op_{\bar m_t}[\ln V(Z^\eps(t), \xi^\eps(t))]dt\bigg|\leq (K_6T+K_7)\eps.
$$
for some $K_7>0$.
The lemma is therefore proved.
\end{proof}

Define $f,g:\R^2_+\mapsto\R$ by
\begin{equation}\label{e-f}
f(x,y)=p_1\big(a_1-b_1x-c_1y\big)+p_2\big(-a_2-b_2y+c_2x\big)
\end{equation}
and
\begin{equation}\label{e-g}
\begin{aligned}
g(x,y)
=&\dfrac{c_2x(\bar a_1-b_1x)+c_1y(\bar a_2-b_2y)}{1+c_2x+c_1y}-\dfrac12\dfrac{a_{11}c_2^2x^2+a_{22}c_1^2y^2+2a_{12}c_1c_2xy}{(1+c_2x+c_1y)^2}.
\end{aligned}
\end{equation}

\begin{lm}\label{lm3.5}
For any $H>0$ and  $k_0>1$, there exist $T_1=T_1(H,\eps_0,k_0)>0$ and $\delta=\delta(H,\eps_0,k_0)>0$
such that for any admissible control $m(\cdot)$,
and $z\in D_{\delta,H}:=([0,H]\times[0,\delta])\cup([0,\delta]\times[0,H])$,
we have
$$
\dfrac1t \int_0^t\E_{z}f(Z(s))ds>9\lambda,\,\text{ and }\,
\dfrac1t \int_0^t\E_{z} g(Z(s))ds\leq\lambda,\,\forall\,t\in[T_1,T_2],$$
and
$$\dfrac1t \int_0^t\E_{z} h(Y(s))ds\leq\dfrac{\lambda}{p_2M},\,\forall\,t\in[T_1,T_2]$$
where $T_2=(k_0+1)T_1$ and $\lambda$ is defined in \eqref{e-lambda}.
\end{lm}

The results in this lemma are obtained by analyzing the behavior
of $Z(t)$ near the boundary.
The proof is postponed to the appendix.


\begin{lm}\label{lm3.6}
With $H, k_0, T_1, T_2,\delta$ as given in Lemma \ref{lm3.5},
there is an $\eps_3>0,\theta\in(0,1)$ such that
for any $\eps\in(0,\eps_3)$.
Let $D^\circ_{\delta,H}=((0,H]\times(0,\delta])\cup((0,\delta]\times(0,H])$.
For any admissible control $m(\cdot)$, $(z,w)\in D_{\delta,H}\times\Sc$,
we have
$$\E_{z,w} \left[V^\eps(Z^\eps(t),\xi^\eps(t))\right]^\theta
\leq e^{-\lambda\theta t}[V^\eps(z,w)]^\theta,  t\in[T_1,T_2].$$
\end{lm}

\begin{proof}
Since $D_{\delta,H}$ is a compact set, by virtue of Lemma \ref{lm3.5}
and \cite[Theorem 5]{KR},
(which tell us we can approximate solutions to \eqref{model-4}
by the corresponding solutions to \eqref{model-5}),
there is an $\eps_2>0$ such that
for any $\eps\in(0,\eps_2)$,
and
for any admissible control $m(\cdot)$, $(z,w)\in D_{\delta,H}\times\Sc$,
we have
\begin{equation}\label{e1-lm3.5}
\dfrac1t \int_0^t\E_{z,w}f(Z^\eps(s))ds>8\lambda,\,\, t\in[T_1,T_2],
\end{equation}
\begin{equation}\label{e2-lm3.5}
\dfrac1t \int_0^t\E_{z,w}g(Z^\eps(s))ds<2\lambda,\,\, t\in[T_1,T_2],
\end{equation}
and
\begin{equation}\label{e3-lm3.5}
\dfrac1t \int_0^t\E_{z} h(Y^\eps(s))ds\leq2\dfrac{\lambda}{p_2M},\,\forall\,t\in[T_1,T_2].
\end{equation}
Note that $f$ and $g$ are not bounded. Thus
 \eqref{e1-lm3.5} and \eqref{e2-lm3.5} do not follow
from the weak convergence of $Z^\eps(\cdot)$ to $Z(\cdot)$.
However, $f$ and $g$ have linear growth rates.
Thus, \eqref{e1-lm3.5} and \eqref{e2-lm3.5}
can still be obtained
from the uniform integrability in Lemma \ref{lm3.3}
combined with the weak convergence.

On the other hand,
\begin{equation}\label{e4-lm3.5}
\begin{aligned}
\op_u \ln V(z,w)=&-f(z)+g(z)-\dfrac{c_1yh(y)u}{1+c_2x+c_1y}+p_2h(y)u\\
\leq&-f(z)+g(z)+Mp_2h(y).
\end{aligned}
\end{equation}
It follows
from \eqref{e1-lm3.5}, \eqref{e2-lm3.5}, \eqref{e3-lm3.5},
and \eqref{e4-lm3.5} that
\begin{equation}\label{e5-lm3.5}
\begin{aligned}
\dfrac1t \int_0^t\E_{z,w} \op_{\bar m_t} \ln V(Z^\eps(s),\xi^\eps(s))ds
\leq -4\lambda,\,\, t\in[T_1,T_2], z\in D^\circ_{\delta,H},\eps<\eps_2
\end{aligned}
\end{equation}
for any admissible control.
In view of \eqref{e5-lm3.5} and Lemma \ref{lm3.4},
when $\eps$ is sufficiently small, we have
\begin{equation}\label{e3-lm3.6}
\E_{z,w}\big[\ln V(Z^\eps(t))\big]-\ln V(z)\leq -3\lambda t,\,\, t\in [T,k_0 T],z\in D^\circ_{\delta,H}.
\end{equation}
Combining \eqref{e3-lm3.6} and \eqref{e3.6},
we have that
$$\E_{z,w}\big[\ln V^\eps(Z^\eps(t))\big]-\ln V^\eps(z)\leq -2\lambda t,\,\, t\in [T,k_0 T],z\in D^\circ_{\delta,H}$$
if $\eps$ is sufficiently small.
Let
$$\Upsilon^\eps_{z,w}(t)=\ln V^\eps(Z^\eps(t),\xi^\eps(t))-\ln V^\eps(Z^\eps(0),\xi^\eps(0)).$$
By \eqref{e3.11} and Lemma \ref{lm3.2},
there is a $\hat K$ depending only on $T_1,T_2$ and $H$ such that
$$\max\Big\{\E_{z,w} \exp(-\Upsilon^\eps(t)), \E_{z,w} \exp(\Upsilon^\eps(t))\Big\}<\hat K,\, z\in D^\circ_{\delta,H},w\in\Sc, t\in[T_1,T_2]$$
for any admissible control.
By Lemma \ref{lm-a0}, there is a $\hat K_2>0$ such that
$$
\begin{aligned}
\ln\left(\E_{z,w} \left[\dfrac{V^\eps(Z^\eps(t),\xi^\eps(t))}{V(z,w)}\right]^\theta\right)=
&\ln\left(\E_{z,w} \exp(\theta\Upsilon^\eps(t))\right)\\
\leq& \theta\E_{z,w} \Upsilon^\eps(t)+\theta^2 \hat K_2\\
\leq& -2\lambda\theta t+\theta^2\hat K_2,\,\,\,\,(z,w)\in D^\circ_{\delta,H}\times \Sc, t\in[T_1,T_2], \theta\in[0,0.5].
\end{aligned}
$$
Letting $\theta=\lambda T_1[\hat K_2]^{-1}\wedge 0.5$,
we have
$$\E_{z,w} \left[V^\eps(Z^\eps(t),\xi^\eps(t))\right]^\theta
\leq e^{-\lambda\theta t}[V^\eps(z,w)]^\theta, (z,w)\in D^\circ_{\delta,H}\times \Sc, t\in[T_1,T_2].$$
\end{proof}

\begin{lm}\label{lm3.7}
Let $\theta$ satisfy the conclusion of Lemma {\rm\ref{lm3.6}}.
There are $q\in(0,1)$ and  $C>0$ such that
$$
\E_{z,w} \left[V^\eps(Z^\eps(T_2),\xi^\eps(T_2))\right]^\theta
\leq q [V^\eps(z,w)]^\theta +C,
$$
for any relaxed Markov control $u^\eps\in {PM}^\eps$ when $\eps$ is sufficiently small.
\end{lm}

\begin{proof}
Applying Jensen's inequality to \eqref{e3.11} and \eqref{e3.12},
we have that for $t\geq0$,
\begin{equation}\label{e1-lm3.7}
\E_{z,w} e^{\theta(t\wedge\tau^\eps)}\left[V^\eps(Z^\eps(t\wedge\tau^\eps),\xi^\eps(t\wedge\tau^\eps))\right]^\theta
\leq [V^\eps(z,w)]^\theta
\end{equation}
and
\begin{equation}\label{e2-lm3.7}
\E_{z,w} \left[V^\eps(Z^\eps(t),\xi^\eps(t))\right]^\theta
\leq e^{(H_1+2)\theta t}[V^\eps(z,w)]^\theta.
\end{equation}
Since $\tilde D_{\delta,H}:=(0,H]^2\setminus D_{\delta,H}$
is a compact subset of $\R^{2,\circ}_+$,
$$C:=e^{(H_1+2)\theta T_2}\sup_{z\in\tilde D_{\delta,H}, w\in\Sc} [V^\eps(z,w)]^\theta <\infty.$$
By virtue of \eqref{e2-lm3.7} and Lemma \ref{lm3.5}, we have
\begin{equation}\label{e3-lm3.7}
\E_{z,w} \left[V^\eps(Z^\eps(t),\xi^\eps(t))\right]^\theta
\leq C+e^{-\theta\lambda} [V^\eps(z,w)]^\theta,\,\forall\, (z,w)\in(0,H]^2\times \Sc, t\in[T_1,T_2].
\end{equation}

We have the following estimate.
\begin{equation}\label{e4-lm3.7}
\begin{aligned}
\E_{z,w}& e^{\theta(T_2\wedge\tau^\eps)}\left[V^\eps(Z^\eps(T_2\wedge\tau^\eps),\xi^\eps(T_2\wedge\tau^\eps))\right]^\theta\\
=&\E_{z,w} \1_{\{\tau^\eps<k_0T_1\}}e^{\lambda(T_2\wedge\tau^\eps)}\left[V^\eps(Z^\eps(T_2\wedge\tau^\eps),\xi^\eps(T_2\wedge\tau^\eps))\right]^\theta\\
&+\E_{z,w} \1_{\{k_0T_1\leq \tau^\eps<T_2\}}e^{\theta\lambda(T_2\wedge\tau^\eps)}\left[V^\eps(Z^\eps(T_2\wedge\tau^\eps),\xi^\eps(T_2\wedge\tau^\eps))\right]^\theta\\
&+\E_{z,w} \1_{\{\tau^\eps\geq T_2\}}e^{\theta\lambda(T_2\wedge\tau^\eps)}\left[V^\eps(Z^\eps(T_2\wedge\tau^\eps),\xi^\eps(T_2\wedge\tau^\eps))\right]^\theta\\
\geq&\E_{z,w} \1_{\{\tau^\eps\leq k_0T_1\}}\left[V^\eps(Z^\eps(\tau^\eps),\xi^\eps(\tau^\eps))\right]^\theta\\
&+e^{\theta\lambda_2k_0T}\E_{z,w} \1_{\{k_0T\leq \tau^\eps<T_2\}}[V^\eps(Z^\eps(\tau^\eps),\xi^\eps(\tau^\eps))]^\theta\\
&+e^{\theta\lambda_2T_2}\E_{z,w} \1_{\{\tau^\eps\geq T_2\}}[V^\eps(Z^\eps(T_2),\xi^\eps(T_2))]^\theta.
\end{aligned}
\end{equation}
With a relaxed Markov control $u^\eps\in{PM}^\eps$,
the process $(Z^\eps(t),\xi^\eps(t))$ is a Markov-Feller process.
Thus, we have from \eqref{e3-lm3.7} that
\begin{equation}\label{e5-lm3.7}
\begin{aligned}
\E_{z,w} &\1_{\{\tau^\eps\leq k_0T_1\}}\left[V^\eps(Z^\eps(T_2),\xi^\eps(T_2))\right]^\theta\\
\leq& \E_{z,w} \1_{\{\tau^\eps\leq k_0T_1\}}[C+e^{-\theta\lambda(T_2-\tau^\eps)}\left[V^\eps(Z^\eps(\tau^\eps),\xi^\eps(\tau^\eps))\right]^\theta\\
\leq& C+ e^{-\theta\lambda T_1}\E_{z,w} \1_{\{\tau^\eps\leq k_0T_2\}}\left[V^\eps(Z^\eps(\tau^\eps),\xi^\eps(\tau^\eps))\right]^\theta.
\end{aligned}
\end{equation}
Similarly, it follows from \eqref{e2-lm3.7} and the inequality $(H_1+2)T\leq \lambda(k_0-1)$ that
\begin{equation}\label{e6-lm3.7}
\begin{aligned}
\E_{z,w} &\1_{\{ k_0T_1\leq \tau^\eps\leq T_2\}}\left[V^\eps(Z^\eps(T_2),\xi^\eps(T_2))\right]^\theta\\
\leq& \E_{z,w} \1_{\{ k_0T_1\leq \tau^\eps\leq T_2\}}e^{\theta(H_1+2)\lambda(T_2-\tau^\eps)}\left[V^\eps(Z^\eps(\tau^\eps),\xi^\eps(\tau^\eps))\right]^\theta\\
\leq& e^{(H_1+2)\theta T_1}\E_{z,w} \1_{\{ k_0T_1\leq \tau^\eps\leq T_2\}}\left[V^\eps(Z^\eps(\tau^\eps),\xi^\eps(\tau^\eps))\right]^\theta\\
\leq& e^{-\theta\lambda T_1}e^{\theta\lambda k_0T_1}\E_{z,w} \1_{\{ k_0T_1\leq \tau^\eps\leq T_2\}}\left[V^\eps(Z^\eps(\tau^\eps),\xi^\eps(\tau^\eps))\right]^\theta.
\end{aligned}
\end{equation}
Moreover,
\begin{equation}\label{e7-lm3.7}
\E_{z,w} \1_{\{\tau^\eps\geq T_2\}}[V^\eps(Z^\eps(T_2),\xi^\eps(T_2))]^\theta
= e^{-\theta\lambda T_1}e^{\theta\lambda T_2}\E_{z,w} \1_{\{\tau^\eps\geq T_2\}}[V^\eps(Z^\eps(T_2),\xi^\eps(T_2))]^\theta.
\end{equation}
Owing to \eqref{e5-lm3.7}, \eqref{e6-lm3.7}, and \eqref{e7-lm3.7}, we have
$$\E_{z,w} [V^\eps(Z^\eps(T_2),\xi^\eps(T_2))]^\theta\leq C+e^{-\theta\lambda T_1}\E_{z,w} e^{\theta(T_2\wedge\tau^\eps)}\left[V^\eps(Z^\eps(T_2\wedge\tau^\eps),\xi^\eps(T_2\wedge\tau^\eps))\right]^\theta.$$
This together with \eqref{e1-lm3.7}
concludes the proof with $q=e^{-\theta\lambda T_1}$.
\end{proof}

\begin{thm}\label{thm3.1}
With $q$ and $C$ given in Lemma \ref{lm3.7},
for sufficiently small $\eps$, we have
\begin{equation}\label{e0-thm3.1}
\E_{z,w}\left[V^\eps(Z^\eps(t),\xi^\eps(t))\right]^\theta
\leq e^{(H_1+2)T_2}q^{t/(2T_2)} [V^\eps(z,w)]^\theta +\dfrac{C}{1-q},
\end{equation}
for any relaxed Markov control $u\in{PM}^\eps$.
\end{thm}

\begin{proof}
By the Markov property,
we have
$$
\E_{z,w} \left[V^\eps(Z^\eps((k+1)T_2),\xi^\eps((k+1)T_2))\right]^\theta
\leq q \E_{z,w}\left[V^\eps(Z^\eps(kT_2),\xi^\eps(kT_2))\right]^\theta +C, k\in\N.
$$
Using this inequality recursively,
we obtain
\begin{equation}\label{e1-thm3.1}
\E_{z,w} \left[V^\eps(Z^\eps(kT_2),\xi^\eps(kT_2))\right]^\theta
\leq q^n [V^\eps(z,w)]^\theta +\dfrac{C(1-q^n)}{1-q}.
\end{equation}
The assertion of this theorem follows from
\eqref{e1-thm3.1} and \eqref{e2-lm3.7}
\end{proof}

\begin{proof}[Proof of Theorem \ref{thm2.2}]
Since
$$\lim_{r\to\infty} \left(\inf_{\{|z|\vee x^{-1}\vee y^{-1}>r, w\in\Sc\}}[V^\eps(z,w)]^\theta\right)=\infty, \
\hbox{ and } \ q<1,$$
the conclusion of Theorem \ref{thm3.1} clearly implies Condition \ref{C1}.
Theorem \ref{thm2.2} is therefore proved.
\end{proof}

\section{Concluding Remarks}\label{sec:rem}

Our main effort in this
 paper is to demonstrate that
we can obtain near-optimal policies for average-cost per unit time yield
for a predator-prey model under fast-varying jump noise
by using a near optimal strategy of
a controlled diffusion model.
Due to the technical complexity of the proofs,
we made some simplifications in the model
in order to facilitate the presentation
but still preserve
important properties of the model.
The main result, Theorem \ref{thm2.2} still hold true if
the following generalizations are made.
\begin{itemize}
\item[(a)] The coefficients $a_i, b_i, c_i, i=1,2$ depend on the state of $\xi^\eps(t)$.
\item[(b)] The wideband noise  in \eqref{model-2}, which is linear in the current setup, can be replaced by nonlinear terms.
\item[(c)] The assumption on $\xi(t)$ in Section 2 can be reduced to the condition that $\xi(t)$ a stationary zero mean process
which is either (i) strongly mixing, right continuous and bounded, with the mixing
rate function $\phi(\cdot)$ satisfying $\int_0^\infty\phi^{1/2}(s)ds<\infty$, or (ii) stationary Gauss-Markov with
an integrable correlation function as in \cite{KR}.
\end{itemize}

With
 the generalization specified in (a) above, the proofs carry over, although the notations are more complicated.
With (b),
we need some additional conditions imposed on the wideband noise parts to obtain
certain boundedness of the solutions to the limit diffusion equation.

Throughout the paper, we assume that $\xi(t)$ is an ergodic Markov process,
under which we can utilize the Fredholm alternative to construct
Lyapunov functions for the wideband noise model \eqref{model-2} based on those for the controlled diffusion \eqref{model-3}.
If that assumption is replaced by (c),
it is slightly more complicated to
construct Lyapunov functions for the wideband noise model \eqref{model-2}.
However, it is doable using the perturbed Lyapunov method in \cite{KR}.
In such a setup, however, we need to work mainly with convergence of probability measures.

In this paper,
we consider the situation that only the predator is harvested.
It is also interesting to deal with the optimization problem of
harvesting both species under the constraint that
the extinction of each species is avoided.
Moreover, time-average optimal harvesting problems
for different ecological models also deserve careful study.
Our methods can be generalized to treat harvested ecological models
of higher dimensions.

\appendix
\section{Appendix}\label{sec:apd}
This appendix provides several technical results.
These results are collected in a number of lemmas.

\begin{lm}\label{lm-a0}
Let $Y$ be a random variable, $\theta_0>0$ a constant, and suppose $$\E \exp(\theta_0 Y)+\E \exp(-\theta_0 Y)\leq K_1.$$
Then the $\log$-Laplace transform
$\phi(\theta)=\ln\E\exp(\theta Y)$
is twice differentiable on $\left[0,\frac{\theta_0}2\right)$ and
$$\dfrac{d\phi}{d\theta}(0)= \E Y,\quad\text{ and }\,0\leq \dfrac{d^2\phi}{d\theta^2}(\theta)\leq K_2\,, \theta\in\left[0,\frac{\theta_0}2\right)$$
 for some $K_2>0$ depending only on $K_1$ and $\theta_0$.
Moreover,
$$
\phi(\theta)\leq\theta\E Y +\theta^2 K_2,\,\,\text{ for }\theta\in[0,0.5\theta_0).
$$
\end{lm}
\begin{proof}
The lemma is proved in \cite{NY17}.
\end{proof}

\begin{lm}\label{lm-a1}
For any $p>0$ and $T>0$, $H>0$,
there is a $\kappa_{p,T}>0$ such that
for any admissible control $m(\cdot)$, we have
\begin{equation}\label{e1-a1}
\E_z(1+|Z_t|)^p\leq \kappa_{p,T}, t\in[0,T], z\in[0,H]^2.
\end{equation}
Moreover,
\begin{equation}\label{e2-a1}
\E_z |X_t|^2\leq x^2\bar\kappa_T^2,
\, \E_z |Y_t|^2\leq y^2\bar\kappa_T^2, t\in[0,T], z\in[0,H]^2.
\end{equation}
\end{lm}

\begin{proof}
By a straightforward computation,
we can show that
$$\op_u (1+c_2x+c_1y)^p\leq C_p(1+c_2x+c_1y)^p,\, z\in\R^2_+.$$
This implies that
$$\E_x (1+c_2X(t)+c_1Y(t))^p\leq e^{C_pt}(1+c_2x+c_1y)^p.$$
Choosing a suitable $\kappa_{p,T}$, we obtain \eqref{e1-a1}.
Now, using the function
$U(z)=(1+c_2x+c_1y)^px^2$,
$$
\begin{aligned}
\op_u U(z)=&U(z)\left[p\dfrac{c_2(a_1-b_1x)+c_1(-a_2-h(y)u-b_2y)}{1+c_2x+c_1y}+2(a_1-b_1x)\right]\\
&+U(z)\bigg[\dfrac{p-1}2\dfrac{a_{11}c_2^2x^2+a_{22}c_1^2y^2+2a_{12}c_1c_2xy}{(1+c_2x+c_1y)^2}\\
&\qquad\qquad+a_{11}+2p\dfrac{(a_{11}+a_{12})c_2x+(a_{12}+a_{22})}{(1+c_2x+c_1y)}\bigg]
\end{aligned}
$$
When $p>0$ is sufficiently large, it can be seen that
there is a $\tilde C_p>0$ satisfying
\begin{equation}\label{e4-a1}
\op_u U(z)\leq\tilde C_pU(z),\,\text{ for }\,z\in\R^{2,\circ}_+.
\end{equation}
Thus,
$$\E_z |X(t)|^2\leq \E_z U(Z(t))\leq U(z)e^{\tilde C_pt}\leq x^2(1+c_2x+c_1y)^pe^{\tilde C_pt}.$$
The above estimate and a similar estimate for $\E_z |Y(t)|^2$
lead to \eqref{e2-a1}.
\end{proof}

Let $f(\cdot)$ and $g(\cdot)$ be defined as in \eqref{e-f} and \eqref{e-g}.
Since $f(z),g(z), h(y)$ are Lipschitz,
there is a $\ell>0$ such that
\begin{equation}\label{e.l1}
|f(z)-f(\tilde z)|\vee|g(z)-g(\tilde z)|\leq \ell(|z-\tilde z|),\, z,\tilde z\in\R^2_+
\end{equation}
and
\begin{equation}\label{e.l2}
|h(y)-h(\tilde y)|\leq \ell(|y-\tilde y|),\, y,\tilde y\geq 0.
\end{equation}

\begin{lm}
Let $\tilde X(t)>0$ and $\tilde Y(t)>0$
satisfy
\begin{equation}\label{tilde-z}
\begin{cases}
d \tilde X(t)=\tilde X(t)\big[\bar a_1-b_1\tilde X(t)\big]dt +\tilde X(t)(\sigma_{11}dW_1(t)+\sigma_{12}dW_2(t))\\
d\tilde Y(t)=\tilde Y(t)\left[-\dfrac{a_2}2-a_{22}-b_2\tilde Y(t)\right]dt + \tilde Y(t)(\sigma_{12}dW_1(t)+\sigma_{22}dW_2(t)).
\end{cases}
\end{equation}
Then there exists a $T_0>0$ such that
\begin{equation}\label{e1-a3}
\dfrac1t\int_0^t\E_x f(\tilde X(s),0)ds\geq 10\lambda;
\,
\dfrac1t\int_0^t\E_x g(\tilde X(s),0)ds\leq \dfrac\lambda2;\,
\end{equation}
 and
 \begin{equation}\label{e2-a3}
 \dfrac1t\int_0^t\E_y \tilde Y(s)ds\leq \dfrac{\lambda}{2(1+M)\ell}
\end{equation}
for any $t>T_0$ and $x, y\in[0,H]$.
\end{lm}

\begin{proof}
Define the occupation measure
$$
\pi^x_t(\cdot)=\int_0^t\PP_x\{\tilde X(s)\in\cdot\}ds
$$
Let $d_1$ and $d_2>0$ be such that
\begin{equation}\label{e4-a3}
x(\bar a_1-b_1x)\leq d_1-x-d_2x^2\,  \text{ for any }x\in\R_+
\end{equation}
Then we have
$$
\begin{aligned}
\E_x \tilde X(t)=&x+\E_x\int_0^t \tilde X(s)\big[\bar a_1-b_1\tilde X(s)\big]ds\\
=&x+d_1t-d_2\E_x\int_0^t \tilde X^2(s)ds,
\end{aligned}
$$
which leads to
\begin{equation}\label{e3-a3}
\E_x\int_0^t \tilde X^2(s)ds
\leq \dfrac {x}{d_2t}+ d_1\leq \dfrac{2H}{d_2t}+d_1,\, x\in[0,H].
\end{equation}
On the other hand,
it follows from It\^o's formula that
$$
\begin{aligned}
\E_x e^t\tilde X(t)=&x+\E_x\int_0^te^s\big(X(s)+X(s)(\bar a_1-b_1X(s))\big)ds\\
\leq&x+d_1\E_z\int_0^te^sds\\
\leq&x+d_1e^t,
\end{aligned}
$$
which implies
\begin{equation}\label{e5-a3}
\E_x \tilde X(t)\leq xe^{-t}+d_1.
\end{equation}
We have from It\^o's formula that
\begin{equation}\label{e6-a3}
\E_x\ln (1+c_2\tilde X(t))=\ln (1+c_2x)+\E_z\int_0^t g(\tilde X(s),0)ds.
\end{equation}
In view of \eqref{e5-a3},
we have that
\begin{equation}\label{e7-a3}
\lim_{t\to\infty}\dfrac{\E_x\ln (1+c_2\tilde X(t))}t=0
\,\text{ uniformly for }\, x\in[0,H].
\end{equation}
Owing to
\eqref{e5-a3} and \eqref{e6-a3}, we have
\begin{equation}\label{e7-a3}
\lim_{t\to\infty}\dfrac1t\E_x\int_0^t g(\tilde X(s),0)ds=0
\,\text{ uniformly for }\, x\in[0,H].
\end{equation}
To proceed, note that the process
$\{\tilde X(t)\}$ has exactly two ergodic invariant probability measures on $\R_+$:
$\bdelta_0$, the Dirac measure concentrated on $0$
and $\mu^*$ on $(0,\infty)$ (see \cite{DNY} for the density of $\mu^*$), while
$\{\tilde Y(t)\}$ admits $\bdelta_0$ as its unique  invariant probability measures on $\R_+$.
Thus, every invariant probability measures $\nu$ of $\{\tilde X(t)\}$
has the form $\nu=\bar\delta\bdelta_0+(1-\bar\delta)\mu^*$ for some $\bar\delta\in[0,1]$.
Direct calculations show that
$$\int_{\R_+}f(x,0)\bdelta(dx)=p_1a_1-p_2a_2>10\lambda,$$
and
$$\int_{\R_+}f(x,0)\mu^*(dx)=p_2\left(-a_2+\dfrac{a_1c_2}
{b_1}\right)>10\lambda.$$
Thus, for any invariant probability measures $\nu$ of $\{\tilde X(t)\}$,
we have
\begin{equation}\label{e8-a3}
\int_{\R_+}f(x,0)\nu(dx)> 10\lambda.
\end{equation}
We now prove that there is a $\tilde T_1>0$
such that
\begin{equation}\label{e9-a3}
\dfrac1t\int_0^t\E_x f(\tilde X(s),0)ds\geq 10\lambda\,
\text{ for }\,t\geq\tilde T_1, x\in[0,H].
\end{equation}
Suppose this claim is false, then we can find
$x_n\in[0,H]$, $t_n>0$ such that
$\lim_{n\to\infty} t_n=\infty$
and $$\int_0^\infty f(\tilde x,0)\pi_{t_n}^{x_n}(d\tilde x)< 10\lambda.$$
In view of \eqref{e3-a3},
the family $\{\pi_{t_n}^{x_n},n\in\N\}$ is tight in $\R_+$.
We can extract a subsequence, still denoted by  $\{\pi_{t_n}^{x_n}\}$, that converges weakly to a probability measure $\nu$.
Since $\pi_t^x$ is the empirical measure of the process $\tilde X(t)$, it is well-known that
$\nu$ is an invariant probability measure on $\R_+$ of the process $\{\tilde X(t)\}$.
By \eqref{e8-a3}  and  the uniform integrability \eqref{e3-a3}.
we must have
$\lim_{n\to\infty}\int_{\R_+}f(\tilde x,0)\pi_{t_n}^{x_n}(d\tilde x)=\int_{\R_+} f(\tilde x,0)\tilde \bnu(d\tilde x)>10\lambda$, which contradicts the assumption.
On the other hand,
Since $\bdelta_0$ is the unique invariant probability measure of $\tilde Y(t)$ in $\R_+$,
similar arguments show that
there exists $\tilde T_2>0$ such that
\eqref{e2-a3} holds for $t\geq\tilde T_2, y\in[0,H]$.
Combining this,  \eqref{e7-a3}, and \eqref{e9-a3}, we obtain the desired results.
\end{proof}

\begin{lm}\label{lm-a4}
For any $H>0$ and $T>0$, there exists a $\tilde\kappa_T$ depending on
$H$ and $T$ such that
\begin{equation}\label{e.1-a4}
\E_{z}|X(t)-\tilde X(t)|\leq \tilde\kappa_T\sqrt{y},
\end{equation}
and
\begin{equation}\label{e.2-a4}
\E_{z}Y(t)\leq\E_z\tilde Y(t)+\tilde\kappa_Tx
\end{equation}
for any $z\in[0,H]^2$ and for any admissible relaxed control.
\end{lm}

\begin{proof}
$$
\begin{aligned}
d[X(t)-\tilde X(t)]=&[X(t)-\tilde X(t)]\big[a_1-b_1(\tilde X(t)+X(t))\big]dt-c_3X(t)Y(t)dt\\
& + [X(t)-\tilde X(t)](\sigma_{12}dW_1(t)+\sigma_{22}dW_2(t)).
\end{aligned}
$$
By H\"older's inequality and \eqref{e1-a1} and \eqref{e2-a1},
we have for any $s\in[0,T]$, $z\in[0,H]^2$ that
$$
\E_z\Big|(X(s)-\tilde X(s))X(s)Y(s)\Big|
\leq\left[ \E_z (X(s)-\tilde X(s))^2\right]^{\frac12}\left[ \E_z  X^4(s))\right]^{\frac14}[\E_z Y^4(s)]^{\frac14}
\leq \tilde\kappa_1 y
$$
for some $\tilde\kappa_1$ depending only on $H$ and $T$.
Applying It\^o's formula yields
$$
\begin{aligned}
\E_z[X(t)-\tilde X(t)]^2
=&\E_z\int_0^t[X(s)-\tilde X(s)]^2\big[2\bar a_1+a_{11}-2b_1(\tilde X(s)
+X(s))\big]ds\\
&-2c_3\E_z\int_0^t(X(t)-\tilde X(s))X(s)Y(s)ds\\
\leq& (2\bar a_1+a_{11})\E_z\int_0^t[X(s)-\tilde X(s)]^2+2c_3\tilde\kappa_1T y.
\end{aligned}
$$
An application of Gronwall's inequality
leads to
$$
\begin{aligned}
\E_z[X(t)-\tilde X(t)]^2
\leq 2c_3\tilde\kappa_1T ye^{(2\bar a_1+a_{11})t}, t\in[0,T].
\end{aligned}
$$
Subsequently, \eqref{e.1-a4} is obtained by applying Holder's inequality to the estimate above.

Let $U(z)$ and $\tilde C_p$ be as in Lemma \ref{lm-a1},
and $\zeta=\inf\left\{t\geq0: U(Z(t))\geq \dfrac{a_2^2}4\right\}.$
If $t<\zeta$ then $X(t)\leq \dfrac{a_2}2$.
Thus,
by a comparison theorem,
$\PP_z\{Y(t)\leq\tilde Y(t), t\leq\zeta\}=1.$
By virtue of It\^o's formula,
$$
\begin{aligned}
\E_z U(Z(t\wedge\zeta))
\leq U(z)+\tilde C_p\E_z\int_0^{t\wedge\zeta}U(Z(s))ds\\
\leq U(z)+\tilde C_p\E_z\int_0^{t}U(Z(s\wedge\zeta))ds.
\end{aligned}
$$
In view of Gronwall's inequality, we have
$$\E_z U(Z(t\wedge\zeta))\leq U(z)e^{\tilde C_pt}.$$
Thus, for $z\in[0,H]^2$, we have
$$\PP_z\{\zeta<T\}\leq (1+c_2x+c_1y)x^2\dfrac{4e^{\tilde C_pt}}{a_2^2}\leq \tilde\kappa_2 x^2$$
for some $\tilde\kappa_2$ depending on $H$ and $T$.
As a result,
$$
\begin{aligned}
\E_z Y(s)ds
\leq& \E_z \1_{\{\zeta<t\}} \tilde Y(t)
+\E_z \1_{\{\zeta<t\}}  Y(t)\\
\leq& \E_z \tilde Y(t) + \left(\PP_z\{\zeta<t\}\E_z\left[ Y(t)\right]^2\right)^{\frac12}\\
\leq&\E_z \tilde Y(t) + \tilde\kappa_4x
\end{aligned}
$$
for some $\tilde\kappa_4$ depending only on $H$ and $T$.
\end{proof}

\begin{proof}[Proof of Lemma \ref{lm3.5}]
We shall show that
$$\dfrac1t \int_0^t\E_{z}f(Z(s))ds\geq9\lambda, t\in[T_1,T_2]$$
when $z\in[0,H]^2$ and either $x$ or $y$ is sufficiently small.
The other claims can be proved similarly.
As a result of \eqref{e.l1}, \eqref{e2-a1}, \eqref{e1-a3} and \eqref{e.1-a4}, we have for $t\in[T_1,T_2]$ that
$$
\begin{aligned}
\dfrac1t \int_0^t\E_{z}f(Z(s))ds
\geq
&\dfrac1t \int_0^t\E_{x}f(\tilde X(s),0)ds
-\left|\dfrac1t \int_0^t\E_{x}f(\tilde X(s),0)ds-\dfrac1t \int_0^t\E_{x}f(Z(s))ds\right|
\\
\geq&\dfrac1t \int_0^t\E_{x}f(\tilde X(s),0)ds-\dfrac\ell t\int_0^t\E_z|\tilde X(s)-X(s)|ds- \dfrac\ell t\int_0^t \E_zY(s)ds\\
\geq&10\lambda-\ell (\tilde\kappa_{T_2}\sqrt{y}+\bar\kappa_{T_2} y)\\
\geq&9\lambda
\end{aligned}
$$
when $z\in[0,H]^2$ and $y$ is sufficiently small.
Similarly, by \eqref{e.l1}, \eqref{e2-a1}, and  \eqref{e.2-a4}, we have
$$
\begin{aligned}
\dfrac1t \int_0^t\E_{z}f(Z(s))ds
\geq
&f(0,0)
-\dfrac1t \int_0^t\E_{z}\left|f(0,0)-f(Z(s))\right|ds
\\
\geq&f(0,0)-\dfrac\ell{t}\int_0^t\E_z\Big(X(s)+Y(s)\Big)ds\\
\geq&f(0,0)-\lambda-\ell (\bar\kappa_{T_2} x+\tilde\kappa_{T_2}x)\\
\geq&9\lambda
\end{aligned}
$$
when $z\in[0,H]^2$ and $x$ is sufficiently small.
Similarly,
using \eqref{e1-a3}, \eqref{e2-a3},
and \eqref{e.l2},
we can prove the remaining results of Lemma \ref{lm3.5}.
\end{proof}

\begin{proof}[Proof of Lemma \ref{lm2.3}]
Similar to \eqref{e4-lm3.3},
we have
$$\op_u V_2(z)\leq K_5-\beta|z|^2,\, \text{ for any }\,z\in\R^2_+.$$
Thus,
\begin{equation}\label{e3-lm2.3}
\dfrac{V_2(Z(t))}t\leq\dfrac{V_2(z)}t+K_5t-\beta\dfrac1t\int_0^t|Z(s)|^2ds
+ \dfrac{\wdt M(t)}t
\end{equation}
where
$$\wdt M(t)=\int_0^t\Big[c_2X(s)(\sigma_{11}dW_1(s)+\sigma_{12}dW_2(s))+c_1Y(t)(\sigma_{12}dW_1(t)+\sigma_{22}dW_2(t))\Big].$$
By the strong law of large number for martingales,
\begin{equation}\label{e4-lm2.3}
\limsup_{t\to\infty}\left( M(t)-\dfrac{\beta}2\int_0^t|Z(s)|^2ds\right)\leq 0 \ \hbox{ a.s.}
\end{equation}
Since
$\liminf_{t\to\infty}\frac{V_2(Z(t))}t\geq0$,
it follows from \eqref{e3-lm2.3} and \eqref{e4-lm2.3}
we obtain that
$$\limsup\dfrac1t\int_0^t|Z(s)|^2ds\leq \dfrac{2K_5}\beta
\ \text{ a.s.}$$
Since $\Phi(yh(y)u))\leq K_6(1+y)$
for some $K_6>0$,
we can easily obtain the desired results: \eqref{e1-lm2.3} and \eqref{e2-lm2.3}.

To show the tightness of the family $\{\eta_\nu:\,\nu\in\Pi_{RM}\}$ in $\R^{2,\circ}_+$,
we can
derive
$$\int_{\R^{2,\circ}_+} V(z)\eta_\nu(dz)\leq \dfrac{C}{1-q}$$
from
the estimate
$$
\E_z [V(Z(t))]^\theta\leq e^{(H_1+2)T_2}q^{t/(2T_2)}[V(z)]^\theta +\dfrac{C}{1-q}
$$
where $\theta$ and $q$ are constants in Theorem \ref{thm3.1}.
The above estimate can be proved in the same manner as
in the proof of Theorem \ref{thm3.1}
with the perturbed Lyapunov function $V^\eps(z,w)$
replaced with $V(z)$.
Alternatively,  it can be shown simply by
letting $\eps\to0$ in \eqref{e0-thm3.1}.
\end{proof}


{\small
\begin{thebibliography}{99}

\setlength{\baselineskip}{0.10in}

\parskip=0pt
\bibitem{LA} L. H. R. Alvarez,  Singular stochastic control in the presence of a state-dependent yield structure, {\it Stochastic Process. Appl.}  86(2006), 323-343.

\bibitem{ALO} L. H. R. Alvarez, E. Lungu, B. \O ksendal,
Optimal multi-dimensional stochastic harvesting with density-dependent prices, {\it Afr. Mat.} 27 (2016), no. 3-4, 427-442.

\bibitem{AS} L.G.
Anderson,  J.C. Seijo,  {\it Bioeconomics of fisheries management.} 2010 John Wiley \& Sons.

\bibitem{ABG} A. Arapostathis, V.S. Borkar, M.K. Ghosh, {\it Ergodic control of diffusion processes,} Vol. 143. Cambridge University Press, 2012.

\bibitem{AHL} C. Azar, J. Holmberg, K. Lindgren,  Stability analysis of harvesting in a predator-prey model. {\it Journal of Theoretical Biology}, {\bf174} (1995), no. 1, 13-19.

\bibitem{BC} J.R. Beddington and J.G. Cooke,  Harvesting from a prey-predator complex ,{\it Ecol. Modelling,}
{\bf14} (1982), 155–177.

\bibitem{BL}
M. Bena{\"\i}m and C. Lobry, Lotka Volterra in fluctuating environment
  or ``how switching between beneficial environments can make survival
  harder'', {\it Ann. Appl. Probab.} 26 (2016), no. 6, 3754-3785.


\bibitem{BP} G. Blankenship, G.C. Papanicolaou, Stability and control of stochastic systems with wide-band noise disturbances. I, {\it SIAM J. Appl. Math.} 34 (1978), no. 3, 437-476.

\bibitem{CC} C.W. Clark, {\it Mathematical Bioeconomics: The Mathematics of Conservation}, Vol. 91., 2010, John Wiley \& Sons.

\bibitem{MC} M.I. da Silveira Costa, Predator harvesting in stage dependent predation models: Insights from a threshold management policy. {\it Mathematical biosciences}, {\bf216} (2008), no. 1, 40-46.

\bibitem{DN} N.H. Du, D.H. Nguyen, Asymptotic behavior of Kolmogorov systems with predator-prey type in random environment. {\it Commun. Pure Appl. Anal.} 13 (2014), no. 6, 2693-2712.

\bibitem{DNY} N.H. Du, D.H. Nguyen, G. Yin, Conditions for permanence and ergodicity of certain stochastic predator-prey models, {\it J. Appl. Probab.} 53 (2016), no. 1, 187-202.

    \bibitem{FR75} W.H. Fleming and R. Rishel, {\it Deterministic and Stochastic Optimal Control}, Springer, New York, 1975.

\bibitem{HNY}
A. Hening, D. Nguyen, and G. Yin,
Stochastic population growth in spatially heterogeneous environments: The density-dependent case, to appear in {\sl J. Math. Biology}.

\bibitem{RK}
R.Z. Khasminskii, \emph{Stochastic stability of differential equations}, second
  ed., Stochastic Modelling and Applied Probability, vol.~66, Springer,
  Heidelberg, 2012, With contributions by G. N. Milstein and M. B. Nevelson.


\bibitem{HK} H.J. Kushner, Modeling and approximations for stochastic systems with state-dependent singular controls and wide-band noise, {\it SIAM J. Control Optim.} 52 (2014), no. 1, 311-338.

\bibitem{KR} H.J. Kushner, W. Runggaldier, Nearly optimal state feedback controls for stochastic systems with wideband noise disturbances, {\it SIAM J. Control Optim.} 25 (1987), no. 2, 298-315.

\bibitem{LB16} M. Liu, C. Bai, Chuanzhi Analysis of a stochastic tri-trophic food-chain model with harvesting, {\it  J. Math. Biol.} 73 (2016), no. 3, 597-625.

\bibitem{LB14} M. Liu, C. Bai, Optimal harvesting policy for a stochastic predator-prey model, {\it Appl. Math. Lett.} 34 (2014), 22-26.

\bibitem{LO} E.M. Lungu, B. \O ksendal, Optimal harvesting from interacting
populations in a stochastic environment, {\it Bernoulli,} 7(3) (2001), 527-539.

\bibitem{LM} Q. Luo, X. Mao, Stochastic population dynamics under regime switching. II, {\it J. Math. Anal. Appl.} 355 (2009), no. 2, 577-593.

\bibitem{MM} E.J.
Milner-Gulland, R. Mace, {\it Conservation of Biological Resources}, 2009, John Wiley \& Sons.

\bibitem{NY17} D.H. Nguyen, G. Yin, Stability of regime-switching diffusion systems with
discrete states belonging to a countable set, {\it submitted}.

\bibitem{RR}  R. Rudnicki, Long-time behaviour of a stochastic prey-predator model, {\it Stochastic Process. Appl.} 108 (2003), no. 1, 93-107.

\bibitem{SSZ} Q. Song, R.H. Stockbridge, C. Zhu, On optimal harvesting problems in random environments, {\it SIAM J. Control Optim.} 49 (2011), no. 2, 859-889.

\bibitem{TY} K. Tran, G. Yin, Optimal harvesting strategies for stochastic competitive Lotka-Volterra ecosystems, {\it Automatica}, 55 (2015), 236-246.

\bibitem{XLH} D.M. Xiao, W.X. Li, M.A. Han,  Dynamics in a ratio-dependent predator-prey model with predator
harvesting. {\it J. Math. Anal. Appl.} {\bf324}(2006), 14–29

\bibitem{ZY} C. Zhu, G. Yin,
On strong Feller, recurrence, and weak stabilization of regime-switching diffusions
{\it SIAM J. Control Optim.} {\bf 48} (2009),  2003--2031.




\end{thebibliography}
}
\end{document}